\documentclass[11pt]{amsart}
\usepackage{amsmath}
\usepackage{amsfonts}
\usepackage{amsthm}
\usepackage{comment}
\usepackage{url}
\usepackage{color}
\usepackage{enumitem}
\setlist[enumerate]{font=\normalfont}
\usepackage{mathtools}
\usepackage[letterpaper,hmargin=1.2in]{geometry}
\usepackage{latexsym}
\usepackage{amssymb}
\usepackage{graphicx}        
\usepackage{float}
\usepackage{tikz-cd}
\usepackage[colorinlistoftodos]{todonotes}
\usepackage[numbers,square]{natbib}
\usepackage[colorlinks=true, allcolors=blue,backref=page,citecolor=blue]{hyperref}
\usepackage{subcaption}
\usepackage{eucal}
\setcitestyle{numbers,open={[},close={]}}

\theoremstyle{definition}
\newtheorem{theorem}{Theorem}[section]
\newtheorem{definition}[theorem]{Definition}

\newtheorem{lemma}[theorem]{Lemma}
\newtheorem{proposition}[theorem]{Proposition}
\newtheorem{corollary}[theorem]{Corollary}

\newtheorem*{theorem*}{Theorem}

\theoremstyle{remark}
\newtheorem{remark}[theorem]{Remark}
\newtheorem{example}[theorem]{Example}

\usepackage{wasysym}

\newcommand{\GF}{\operatorname{GF}}
\newcommand{\Berg}{\operatorname{Berg}}

\renewcommand{\phi}{\varphi}

\newcommand{\del}{\partial}

\title{\textbf{Toric vector bundles with trivial Chern class and flag decorations}}
\author{Sergio Cristancho}
\address{Princeton University}
\email{sergio.cris@princeton.edu}

\begin{document}

\begin{abstract}
	We study toric vector bundles on complete toric varieties whose total equivariant Chern classes are trivial. Our approach is tropical, using the notion of tropical toric vector bundles as piecewise linear maps introduced by Kaveh and Manon. We prove that any toric vector bundle of rank $r$ with trivial Chern class and affinely independent equivariant Chern roots is equivariantly isomorphic to a toric vector bundle pulled back from one of a finite set of varieties with dimension at most $r-1$ after twisting by a character. This extends a theorem of Payne about toric vector bundles of rank $r\leq 3$ with trivial Chern class. As an application, we construct examples of complete toric varieties of dimension $n$ that admit no nontrivial toric vector bundles of rank $r\leq n+1$ with the aforementioned properties. We also introduce combinatorial gadgets we call flag decorations of permutohedra, whose convexity properties are key for our results. 
\end{abstract}
\maketitle
 
\section{Introduction}

Torus equivariant vector bundles on a toric variety, \emph{toric vector bundles}, have various combinatorial descriptions. The theory of toric line bundles, for instance, is very well understood in terms of support functions and polyhedra. Determining whether a complete toric variety $X$ admits a nontrivial toric line bundle reduces to a linear algebra problem. For toric vector bundles of higher rank, the seminal works of Kaneyama~\cite{kaneyama_equivariant_1975} and Klyachko~\cite{klyachko_equivariant_1990} characterized them in terms of linear algebraic data compatible with the fan structure. They have also been described in more convex geometric fashion by Di Rocco, Jabbusch and Smith~\cite{rocco_toric_2018} using parliaments of polytopes, and by Kaveh and Manon~\cite{kaveh_toric_2023} in terms of piecewise linear valuations or piecewise linear maps. Despite these advances, the analogous question 
\begin{quote}
	\centering
	\emph{Given a complete toric variety $X$, does it admit a nontrivial toric vector bundle?} 
\end{quote}
remains difficult to answer in general. We say a toric vector bundle is \emph{trivial} if it is equivariantly isomorphic to a $T$-linearized free vector bundle. 

\subsection{Toric vector bundles with trivial Chern class} In~\cite{payne_toric_2008}, Payne studies the previous question by developing the theory of branched covers of fans. Using this point of view, he exhibits toric threefolds that do not admit any nontrivial toric vector bundles of rank $r\leq 3$. His analysis requires special treatment of toric vector bundles whose total equivariant Chern class is globally polynomial, with \emph{trivial Chern class}. He proves that a toric vector bundle of rank $r\leq 3$ with trivial Chern class on a complete toric variety $X$ is either trivial or equivariantly isomorphic to a bundle pulled back from $\mathbb{P}^1$ or one of a finite list of projective toric surfaces after twisting by a character. Our first observation is that Payne's methods can be interpreted using the language of piecewise linear functions of Kaveh and Manon~\cite{kaveh_toric_2023} and furthermore depend only on the \emph{tropicalizations} of the bundles. 

Recently, the notion of \emph{tropical toric vector bundles} (tropical bundles for short) was introduced independently by Khan and Maclagan~\cite{khan_tropical_2024} and by Kaveh and Manon~\cite{kaveh_tropical_2024}. Both approaches replace the linear algebraic data by \emph{matroid} data that is purely combinatorial. In this paper we take the point of view of Kaveh and Manon~\cite{kaveh_tropical_2024}, whereby a tropical bundle $\mathfrak{E}$ of rank $r$ on a fan $\Delta \subset N_\mathbb{R}$ is presented as a \emph{piecewise linear map} $\Phi:|\Delta|\to \Berg(\mathrm{M})$ from the support of $\Delta$ to the \emph{Bergman fan} $\Berg(\mathrm{M})$ of a matroid $\mathrm{M}$ of rank $r$. When $\mathfrak{E}$ has trivial Chern class, it posseses a well-defined multiset of global equivariant Chern roots.

Armed with this tropical perspective, we set out to test this strategy for bundles of higher rank. Our first finding is the following:

\begin{theorem}[Tropical bundle version]\label{thm:trivial_chern}
Let $\mathfrak{E}$ be a tropical bundle of rank $r$ on a complete fan $\Delta\subset N_\mathbb{R}$ with piecewise linear map $\Phi:|\Delta|\to \Berg(\mathrm{M})$. Suppose that $\mathfrak{E}$ has trivial Chern class and affinely independent Chern roots ${\boldsymbol{u}}=\lbrace u_1,\dots,u_r \rbrace$. Let $U:|\Delta| \to \mathbb{R}^r/\mathbb{R}\boldsymbol{1}$ be the linear map given by $U(x)=(\langle u_1, x \rangle,\dots,\langle u_r, x \rangle)$. Then there exist the following:
\begin{enumerate}
	\item a complete fan $\Xi$ that coarsens the $(r-1)$-dimensional permutohedral fan $\Sigma_r$,
	\item a tropical bundle $\mathfrak{F}$ of rank $r$ on $\Xi$ with PL map $\Psi: |\Xi|\to \Berg(\mathrm{M})$ and 
	\item a character $u\in M$ 
\end{enumerate}
such that ${U}: \Delta \to \Xi$ is a surjective morphism of fans and the following diagram commutes.
\[
\begin{tikzcd}
	\lvert\Delta\rvert \arrow[r,"{U}"] \arrow[rd,"\Phi\otimes u"'] & \lvert \Xi \rvert \arrow[d,"\Psi"]\\
	{} & \Berg(\mathrm{M})
\end{tikzcd}
\] 
In other words, $\mathfrak{E}\otimes \mathfrak{O}(u)$ is isomorphic to ${U}^* \mathfrak{F}$, where $\mathfrak{O}(u)$ is the tropical line bundle determined by $u:|\Delta|\to \mathbb{R}$. 
\end{theorem}

Recall that the $(r-1)$-dimensional \emph{permutohedral fan} $\Sigma_r \subset \mathbb{R}^r/\mathbb{R}\boldsymbol{1}$ is the fan cut out by all hyperplanes of the form $x_i-x_j=0$ for $i<j$. Equivalently, it is the normal fan of the $r-1$ dimensional \emph{permutohedron} $\Pi_r \subset \mathbb{R}^r$, which is the Minkowski sum of all segments of the form $[e_i,e_j]$ for $i<j$, where $e_i$ is the $i$-th standard basis vector in $\mathbb{R}^r$. 

Here we allow fans with nonzero lineality space, and we say that a fan is \emph{trivial} when it is comprised of a single linear space. The proof of Theorem~\ref{thm:trivial_chern} constructs $\Xi$ and $\mathfrak{F}$ from $\mathfrak{E}$ explicitly, which allows us to determine when $\mathfrak{E}$ is trivial.

\begin{corollary}\label{cor:trivial}
	The tropical bundle $\mathfrak{E}$ is trivial if and only if the fan $\Xi$ provided by the proof of Theorem~\ref{thm:trivial_chern} is trivial.
\end{corollary}

The previous tropical statement implies its analogue for vector bundles, which extends Payne's results when the {Chern roots} of the toric vector bundle are affinely independent.

\begin{theorem}[Vector bundle version]\label{thm:analogous}
Let $\mathcal{E}$ be a toric vector bundle of rank $r$ on a complete toric variety $X_\Delta$ with fan $\Delta \subset N_\mathbb{R}$. Suppose that $\mathcal{E}$ has trivial Chern class and affinely independent Chern roots $\boldsymbol{u}=\lbrace u_1,\dots,u_r \rbrace$. Let $U:|\Delta| \to \mathbb{R}^r/\mathbb{R}\boldsymbol{1}$ be the linear map given by $U(x)=(\langle u_1, x \rangle,\dots,\langle u_r, x \rangle)$. Then there exist the following:
\begin{enumerate}
	\item a complete fan $\Xi$ that coarsens the $(r-1)$-dimensional permutohedral fan $\Sigma_r$,
	\item a toric vector bundle $\mathcal{F}$ of rank $r$ on $X_\Xi$,
	\item and a character $u\in M$
\end{enumerate}
such that $\pi_{{U}}:X_\Delta\to X_\Xi$ is a surjective toric morphism and $\mathcal{E}\otimes \mathcal{O}(\text{div}\chi^u)$ is equivariantly isomorphic to $\pi_{{U}}^* \mathcal{F}$. 
\end{theorem}

The previous theorems only apply when the rank of the bundle is at most the dimension of the fan plus one due to the assumption of affine independence. It is natural to ask whether affine independence of the Chern roots in Theorem~\ref{thm:trivial_chern} is necessary. Alas, the natural generalization of Theorem~\ref{thm:trivial_chern} dropping the independence assumption is false, see Section~\ref{ssec:failure}. 

\subsection{Applications} We provide two applications of Theorem~\ref{thm:analogous} as an obstruction to the existence of toric vector bundles with the relevant properties.

\begin{theorem}\label{thm:rank4}
	If $X_\Delta$ is a complete toric variety with no nonconstant maps to projective space, then $X_\Delta$ admits no nontrivial toric vector bundles of rank $r\leq 4$ with trivial Chern class and affinely independent equivariant Chern roots.
\end{theorem}

\begin{theorem}\label{thm:existence}
	For each $n\geq 3$ there exists a complete toric variety of dimension $n$ with no nontrivial toric vector bundles of rank $r \leq n+1$ with trivial Chern class and affinely independent equivariant Chern roots.
\end{theorem} 

Theorem~\ref{thm:rank4} follows from Studen\'y's~\cite{studeny_probabilistic_2005} classification of fan coarsenings of the permutohedral fan for $r\leq 4$. Meanwhile, we deduce Theorem~\ref{thm:existence} by studying a concrete family of complete toric varieties originally studied by Perling and Schröer~\cite{perling_vector_2016}, see Example~\ref{ex:existence}.

\subsection{Flag decorations of permutohedra and sweep polytopes}  A \emph{sweep polytope} is a zonotope that is equal to the Minkowski sum of all segments of the form $[v_i,v_j]$ for a fixed collection of vectors $v_1,\dots,v_r$. Sweep polytopes were thoroughly studied by Padrol and Philippe~\cite{padrol_sweeps_2024}, but they appear under many names in the literature of permutation combinatorics and convex optimization. The chief example of a sweep polytope is the familiar permutohedron $\Pi_r$. 

The key ingredients of our proof are combinatorial gadgets we call \emph{flag decorations} on {sweep polytopes}. A flag decoration assigns a complete flag of a matroid to each vertex of a sweep polytope in a compatible fashion. We will be interested in the convexity properties of these flag decorations, principally on whether they induce fan coarsenings on the normal fans of the sweep polytope they decorate. Theorem~\ref{thm:permuto} shows that flag decorations on permutohedra always induce fan coarsenings of the corresponding permutohedral fan. This is not necessarily true for flag decorations of general sweep polytopes. 

We also show that any tropical bundle with trivial Chern class induces a flag decoration of a sweep polytope. The flags appearing in this decoration are precisely those induced by the \emph{Klyachko data} of the associated bundle. This is the source of the fan coarsenings appearing in Theorem~\ref{thm:trivial_chern}.  

\subsection{Organization} Section~\ref{sec:setup} gives a brief introduction to toric and tropical bundles and the necessary matroid theory. In Section~\ref{sec:decorations} we define flag decorations and study certain convexity properties. We prove Theorems~\ref{thm:trivial_chern} and~\ref{thm:analogous} in Section~\ref{sec:proofs} and discuss the necessity of the independence assumption. Finally, Section~\ref{sec:examples} is dedicated to the applications of our main theorem, Theorems~\ref{thm:rank4} and~\ref{thm:existence}.

\subsection*{Acknowledgments} The author thanks Matt Larson for helpful discussions and June Huh for bringing Payne's work~\cite{payne_toric_2008} to their attention.  

\subsection*{AI Disclosure} ChatGPT Pro 5.6 was used in the development of this paper for the following purposes: literature search, proofreading, producing figures, finding Examples~\ref{ex:non_convex_deco} and~\ref{ex:no_coarsening}, and computations for Example~\ref{ex:existence}.

\section{Toric vector bundles and matroids }\label{sec:setup}

In this section we collect the preliminaries on toric vector bundles and matroids necessary for the upcoming sections. Throughout we try to use standard terminology in the matroid, convex geometry and toric variety literature; we will point out when it deviates. As is customary in matroid theory, we omit the braces when writing singletons. We also denote the ray and line spanned by a single vector $v$ as $\mathbb{R}_{\geq 0}v$ and $\mathbb{R}v$ respectively. 

For further reading on the geometry of matroids we refer the reader to the survey by Ardila-Mantilla~\cite{ardila-mantilla_geometry_2023}, while for toric geometry we refer to the textbooks by Fulton~\cite{fulton_introduction_2016} and by Cox, Little and Schenck~\cite{cox_toric_2011}.

\subsection{The Bergman fan of a matroid} Let $\mathrm{M}$ be a loopless matroid of rank $r$ on a ground set $\mathrm{E}$. We denote its lattice of flats by $\Lambda(\mathrm{M})$. The set $\operatorname{Fl}(\mathrm{M})$ of flags of flats in $\Lambda(\mathrm{M})$ is partially ordered by \emph{refinement}. 

In the vector space $\mathbb{R}^\mathrm{E}$, $\{  e_s \}_{s\in \mathrm{E}}$ denotes the standard basis, $e_S$ denotes $\sum_{s\in S}e_s$, and $e_s^*$ and $e_S^*$ are defined similarly in the dual space $(\mathbb{R}^\mathrm{E})^*$. 

The \emph{Bergman fan} $\Berg(\mathrm{M})\subset \mathbb{R}^\mathrm{E}$ of $\mathrm{M}$ is the polyhedral fan composed of all the cones 
\[
C(\mathcal{F}) = \operatorname{cone}( e_{F^1}, e_{F^2}, \dots , e_{F^{k-1}})+ \mathbb{R} \boldsymbol{1}
\] 
where 
\[
\mathcal{F}: \varnothing \subset F^1 \subset \dots \subset F^k =\mathrm{E}
\] 
is a flag of flats of $\mathrm{M}$. The face poset of $\Berg(\mathrm{M})$ is isomorphic to the poset $\operatorname{Fl}(\mathrm{M})$.

\begin{remark}
 	In the literature, the previous definition is often called the \emph{lifted} Bergman fan, while the usual definition of Bergman fan refers to the projection of the previous fan on $\mathbb{R}^{\mathrm{E}}/\mathbb{R}\boldsymbol{1}$. In this paper, we will refer to the lifted fan as the Bergman fan for simplicity.
\end{remark}

The \emph{base polytope} $\mathrm{P}(\mathrm{M})\subset (\mathbb{R}^\mathrm{E})^*$ of $\mathrm{M}$ is the convex hull of all points of the form $e_B^*$ where $B$ is a basis of $\mathrm{M}$. The \emph{Gröbner fan} ${\GF}(\mathrm{M})\subset \mathbb{R}^\mathrm{E}$ of $\mathrm{M}$ is the normal fan of $\mathrm{P}(\mathrm{M})$. If $B$ is a basis of $\mathrm{M}$, denote by $\sigma(B)$ the corresponding maximal cone in $\GF(\mathrm{M})$. 

The support of $\Berg(\mathrm{M})$ is equal to the support of a subfan of $\GF(\mathrm{M})$. We say that a flag $\mathcal{F}$ of flats of $\mathrm{M}$ is adapted to a basis $B$ if each flat of $\mathcal{F}$ is the span of a collection of elements of $B$. Every flag is adapted to some basis of $\mathrm{M}$. If $B$ is a basis of $\mathrm{M}$, we call $\mathcal{A}(B,\mathrm{M})=\Berg(\mathrm{M})\cap \sigma(B)$ the \emph{apartment} of $B$ in $\Berg(\mathrm{M})$. Any apartment $\mathcal{A}(B,\mathrm{M})$ is  piecewise linearly isomorphic to $\mathbb{R}^B\cong \mathbb{R}^r$ by means of the projection map $\pi_B: \mathcal{A}(B,\mathrm{M})\to \mathbb{R}^B$ and the piecewise linear inverse defined coordinatewise:
\begin{align}
	\phi_{B} (w)_i &= \begin{cases}
	w_i &\text{ if }i\in B,\\
	\min  \lbrace w_j \mid j \in C_i(B)\setminus i \rbrace  &\text{ otherwise},
\end{cases} \label{eqn:pl_iso}
\end{align}
where $C_i(B)$ is the fundamental circuit of $i$ with respect to $B$. 

\subsection{Toric and tropical vector bundles} 

We now give a brief introduction to the theory of toric and tropical vector bundles on toric varieties. We follow the recent work of Kaveh and Manon~\cite{kaveh_toric_2023}~\cite{kaveh_tropical_2024} and refer to their papers for some of the missing definitions.

Let $X_\Delta$ be a toric variety with rational polyhedral fan $\Delta$ with respect to the lattice $N$ with dual lattice $M$. We allow $\Delta$ to have a nonzero lineality space $N_0$, in which case $X_\Delta$ is the toric variety associated to the strongly convex fan $\Delta/N_0$ but we retain the action by the larger torus associated to $N$. 

A toric vector bundle $\mathcal{E}$ of rank $r$ on the toric variety $X_\Delta$ with fiber $E$ over the identity, $\text{dim}(E)=r$, is given by a \emph{piecewise linear (PL) map} $\tilde{\Phi}: |\Delta| \to \tilde{\mathcal{B}}(E)$ from the support of $\Delta$ to the space of vector space valuations $\tilde{\mathcal{B}}(E)$ on $E$ satisfying the following property: for any cone $\delta$ of $\Delta$, there exists a frame $F$ of $E$ and an integral linear map $T_\delta:\text{span} |\delta|\to \mathbb{R}^r$ such that 
\[
\tilde{\Phi}|_{\delta}=\tilde{\phi}_{F} \circ T_\delta,
\]
where $\tilde{\phi}_F:\mathbb{R}^r\to \mathcal{A}(F,E)$ is the PL isomorphism between $\mathbb{R}^r$ and the apartment of $F$ in $\tilde{\mathcal{B}}(E)$. With this notation, we say that $\tilde{\Phi}$ is adapted to $F$ on $\delta$, and if $T_\delta=(u_1,\dots,u_r)$, we say that $u_1,\dots,u_r\in M$ are \emph{local equivariant Chern roots} of $\mathcal{E}$ at $\delta$. 

A \emph{tropical toric vector bundle} $\mathfrak{E}$ of rank $r$ on the fan $\Delta$, or a tropical bundle for short, is given by a {PL map} $\Phi: |\Delta| \to \Berg(\mathrm{M})$ from the support of $\Delta$ to the Bergman fan of a matroid $\mathrm{M}$ of rank $r$ satisfying the following property: for any cone $\delta$ of $\Delta$, there exists a basis $B$ of $\mathrm{M}$ and an integral linear map $T_\delta:\text{span} |\delta|\to \mathbb{R}^r$ such that 
\[
{\Phi}|_{\delta}={\phi}_{B} \circ T_\delta,
\]
where ${\phi}_B:\mathbb{R}^r\to \mathcal{A}(B,\mathrm{M})$ is the PL isomorphism defined in \eqref{eqn:pl_iso}. With this notation, we say that $\Phi$ is adapted to $B$ on $\delta$, and if $T_\delta=(u_1,\dots,u_r)$, we say that $u_1,\dots,u_r\in M$ are local equivariant Chern roots of $\mathfrak{E}$ at $\delta$. 

Given a tropical bundle $\mathfrak{E}$ with PL map $\Phi: |\Delta| \to \Berg(\mathrm{M})$, a matroid representation $L: \mathrm{M}\to E$ induces an inclusion $L:\Berg(\mathrm{M})\hookrightarrow \tilde{\mathcal{B}}(E)$ which in turn induces a toric vector bundle $\mathcal{E}$ by composing $\Phi$ with this inclusion. 

Conversely, let $\mathcal{E}$ be a toric vector bundle with PL map $\tilde{\Phi}: |\Delta|\to \tilde{\mathcal{B}}(E)$. Let $L: \mathrm{M}\to E$ be a matroid representation satisfying that $\operatorname{span} L(\mathrm{M})=E$ and  that for every $\delta \in \Delta$, there exists a basis $B$ of $\mathrm{M}$ such that $F=L(B)$ is a frame of $E$ which $\tilde{\Phi}$ adapts to. This determines a tropical bundle $\mathfrak{E}$ whose PL map $\Phi$ fits in the following diagram:
\[
\begin{tikzcd}
\lvert\Delta\rvert \arrow[r,"\tilde{\Phi}"] \arrow[rd,"\Phi"']& \tilde{\mathcal{B}}(E)\\
{} & \Berg(\mathrm{M}) \arrow[u,hookrightarrow,"L"']
\end{tikzcd}
\]
Any such choice of $\mathfrak{E}$ is called a \emph{tropicalization} of $\mathcal{E}$. 


\subsection{Toric line bundles and twisting} The space of vector space valuations on a one dimensional vector space and the Bergman fan of any loopless rank 1 matroid are isomorphic to $\mathbb{R}$. Hence the PL maps describing a line bundle $\mathcal{L}$ and its tropicalization $\mathfrak{L}$ are equivalent and given by a piecewise linear function $h:|\Delta|\to \mathbb{R}$ in the usual sense. In this case the classical and tropical descriptions agree exactly. For example, for any character $u\in M$, the line bundle $\mathcal{O}(\text{div}\chi^u)$ and its tropicalization $\mathfrak{O}(u)$ have the same PL map, $u:|\Delta|\to \mathbb{R}$.

Furthermore, one can prove that, if $\mathcal{E}$ is a toric vector bundle with PL map $\tilde{\Phi}:|\Delta|\to \tilde{\mathcal{B}}(E)$, then the PL map $\tilde{\Phi}\otimes h$ associated to $\mathcal{E}\otimes \mathcal{L}$ is given locally as $(\tilde{\Phi}\otimes h)|_{\delta}(x)=\tilde{\phi}_{F_\delta} (T_\delta(x) + h(x) \boldsymbol{1})$ for any $\delta\in \Delta$. Naturally, if $\mathfrak{E}$ is a tropical bundle with PL map $\Phi: |\Delta| \to \Berg(\mathrm{M})$, we can define $\mathfrak{E}\otimes \mathfrak{L}$ similarly. 

\subsection{Trivial vector bundles}\label{ssec:trivial} A toric vector bundle $\mathcal{E}$ on a toric variety $X_{\Delta}$ with fan $\Delta\subset N_\mathbb{R}$ is said to be \emph{trivial} if it is equivariantly isomorphic to a $T$-linearized free vector bundle $\mathcal{O}(\text{div}\chi^{u_1})\oplus \dots \oplus \mathcal{O}(\text{div}\chi^{u_r})$ with $u_1,\dots,u_r\in M$. In the PL map language, $\mathcal{E}$ is trivial if its PL map $\tilde{\Phi}:|\Delta|\to \tilde{\mathcal{B}}(E)$ adapts globally to a single frame, i.e. if $\tilde{\Phi}=\tilde{\phi_F}\circ (u_1,\dots,u_r)$ for some frame $F$ of $E$. 

In analogy, we say that a tropical bundle $\mathfrak{E}$ on a fan $\Delta$ is \emph{trivial} if its PL map ${\Phi}:|\Delta|\to \Berg(\mathrm{M})$ adapts globally to a single basis, i.e. if $\Phi={\phi_B}\circ (u_1,\dots,u_r)$ for some basis $B$ of $\mathrm{M}$. In the tropical setting being trivial is not strictly the same as being equal to $\mathfrak{O}(u_1)\oplus \dots \oplus \mathfrak{O}(u_r)$ but an \emph{extension} of $\mathfrak{O}(u_1)\oplus \dots \oplus \mathfrak{O}(u_r)$, see \cite[Sec. 8]{kaveh_tropical_2024} for a discussion.

\subsection{Equivariant Chern classes}\label{ssec:chern} The \emph{$i$-th equivariant Chern class} $c^T_i(\mathcal{E})$ of a toric vector bundle $\mathcal{E}$ on $X_\Delta$ with PL map $\tilde{\Phi}$ is a piecewise polynomial function on $\Delta$, given by the degree $i$ elementary symmetric polynomial $\epsilon_i:\mathbb{R}^r\to \mathbb{R}$ applied to the local Chern roots of $\mathcal{E}$ on each cone $\delta\in \Delta$. In other words, it is given conewise by $\epsilon_i \circ \pi_\mathcal{L} \circ \tilde{\Phi}|_\delta$ on each $\delta\in \Delta$ if $\delta$ is adapted to a frame $\mathcal{L}$. Analogously, the $i$-th equivariant Chern class $c^T_i(\mathfrak{E})$ of a tropical bundle $\mathfrak{E}$ with PL map $\Phi$ is given conewise by $\epsilon_i \circ \pi_B \circ {\Phi}|_\delta$ on each $\delta\in \Delta$ if $\delta$ is adapted to a basis $B$. 

The \emph{total equivariant Chern class} $c^T$ of a bundle (vector or tropical) is given by
\[
c^T = 1 + c_1^T + \dots + c_r^T
\]
where $r$ is the rank of the corresponding bundle. We say that a bundle has \emph{trivial equivariant Chern class} (abbreviated trivial Chern class) if $c^T$ is globally polynomial on $|\Delta|$. When the fan $\Delta$ is complete, this is equivalent to the existence of a \emph{global} multiset of equivariant Chern roots ${\boldsymbol{u}}=\lbrace u_1,\dots,u_r \rbrace$ for the bundle (toric or tropical), meaning that $T_\delta = (u_1,\dots,u_r)$ for every cone $\delta\in \Delta$.

\section{Flag decorations}\label{sec:decorations}

In this section we introduce our main combinatorial tool, flag decorations. First we define them for permutohedra, and show that they induce fan coarsenings of the permutohedral fan. After that, we define them on general sweep polytopes, but point out they do not enjoy the same convexity properties, which will be relevant in the upcoming sections. 

\subsection{Flag decorations of the permutohedron}\label{ssec:permuto_decorations}

The $r-1$ dimensional \emph{permutohedron} is the zonotope $\Pi_r \subset \mathbb{R}^r$ defined as
\[
\Pi_r = \sum\limits_{i<j} [e_i,e_j]
\] 
and it naturally lives in the affine hyperplane $x_1+\dots +x_r =\frac{r(r-1)}{2}$ perpendicular to the all-ones vector $\boldsymbol{1}$. Its normal fan $\Sigma_r\subset \mathbb{R}^r/\mathbb{R}\boldsymbol{1}$ is called the $r-1$ dimensional \emph{permutohedral fan}. Notably, the fan $\Sigma_r$ is equal to the projected Bergman fan of the Boolean matroid on $r$ elements. The cones of $\Sigma_r$, alternatively the faces of $\Pi_r$, are in correspondence with flags of subsets of $[r]$. In fact, $\Sigma_r$ is composed of all the polyhedral cones of the form 
\[
C(\mathcal{W})=\operatorname{cone}(e_{W^1}, e_{W^2}, \dots , e_{W^{k-1}})+\mathbb{R}\boldsymbol{1}
\]
where 
\[ 
\mathcal{W}: \varnothing \subset W^1 \subset \dots \subset W^{k-1} \subset W^k=[r]
\] 
is a flag of subsets of $[r]$. Note that a cone $C(\mathcal{V})$ is a face of $C(\mathcal{W})$ precisely when $\mathcal{W}$ is a {refinement} of $\mathcal{V}$. In particular, the face poset of $\Sigma_r$ is isomorphic to the poset $\operatorname{Fl}_r$ of flags of subsets of $[r]$. 

\begin{definition}
	Let $\mathrm{M}$ be a rank $r$ matroid on $\mathrm{E}$ and let $\operatorname{Fl}(\mathrm{M})$ be its poset of flags of flats. A \emph{flag decoration} of the permutohedron $\Pi_r$ by flags of $\mathrm{M}$ is a map of posets $F:{\operatorname{Fl}}_r \to \operatorname{Fl}(\mathrm{M})$ satisfying that for any flag 
	\[ 
	\mathcal{W}\colon{\varnothing \subset W^1 \subset \dots \subset W^{k-1} \subset W^k=[r]}
	\]
	we have that 
	\[
	F(\mathcal{W})\colon{\varnothing \subset F(\mathcal{W})^1 \subset \dots \subset F(\mathcal{W})^{k-1} \subset F(\mathcal{W})^k=\mathrm{E}}
	\] 
	has the same \emph{rank shape} as $\mathcal{W}$, meaning $\text{rk}_\mathrm{M} F(\mathcal{W})^i= |W^i|$ for every $1\leq i \leq k$.
\end{definition}
\begin{remark}
	For any nonempty proper subset $S\subset [r]$, if $F( \varnothing \subset S\subset [r] )=\varnothing\subset F(S) \subset \mathrm{E}$ we have that $\text{rk}_\mathrm{M}(F(S))=|S|$. Since $F$ preserves cover relations, it follows that for any flag 
	\[
	\mathcal{W}\colon{\varnothing \subset W^1 \subset \dots \subset W^{k-1} \subset W^k=[r]}
	\] 
	we have that 
	\[ 
	F(\mathcal{W})\colon{\varnothing \subset F(W^1) \subset \dots \subset F(W^{k-1}) \subset F(W^k)=\mathrm{E}}.
	\] 
\end{remark}

\subsubsection{Convex rank tests}

A fan $\Xi$ coarsening the permutohedral fan $\Sigma_r$ is given by a collection of edges $\mathcal{E}$ of the permutohedron whose endpoints correspond to maximal cones of $\Sigma_r$ that are merged to obtain $\Xi$. Morton, Pachter, Shiu, Sturmfels and Wienand~\cite{morton_convex_2009} studied the collections of edges $\mathcal{E}$ of $\Pi_r$ that determine fans that coarsen $\Sigma_r$ under the name \emph{convex rank tests}. They prove that such a collection determines a fan $\Xi=\Xi(\mathcal{E})$ coarsening $\Sigma_r$ if and only if $\mathcal{E}$ satisfies both of the following properties:
\begin{itemize}
	\item[$\square$] \emph{The square property:} For every square 2-face $P$ of $\Pi_r$, if one edge of $P$ is in $\mathcal{E}$, the opposite edge of $P$ is also in $\mathcal{E}$.  
	\item[$\hexagon$] \emph{The hexagon property:} For every hexagonal 2-face $P$ of $\Pi_r$, if two consecutive edges of $P$ are in $\mathcal{E}$, the opposite two edges of $P$ are also in $\mathcal{E}$.  
\end{itemize}
See Figure~\ref{fig:properties} for a visualization.

\begin{figure}
	\[
	\begin{tikzpicture}[scale=2.5,vertex/.style={circle, fill=black, inner sep=1.5pt}]
	\coordinate (A) at (0,0);
	\coordinate (B) at (1,0);
	\coordinate (C) at (1,1);
	\coordinate (D) at (0,1);
	
	\draw (A)--(B)--(C)--(D)--(A);
	\draw[red,dashed,line width=2pt] (A) -- (D);

	\node[vertex] at (A) {};
	\node[vertex] at (B) {};
	\node[vertex] at (C) {};
	\node[vertex] at (D) {};

	\node[] at (1.5,0.5) {\large$\Longrightarrow$};

	\coordinate (A') at (2,0);
	\coordinate (B') at (3,0);
	\coordinate (C') at (3,1);
	\coordinate (D') at (2,1);
	
	\draw (A')--(B')--(C')--(D')--(A');
	\draw[red,dashed,line width=2pt] (A') -- (D');
	\draw[red,dashed,line width=2pt] (B') -- (C');

	\node[vertex] at (A') {};
	\node[vertex] at (B') {};
	\node[vertex] at (C') {};
	\node[vertex] at (D') {};
	\end{tikzpicture}
	\]
	\vspace{0.5cm}
	\[
	\begin{tikzpicture}[
	x={(1cm,0cm)},
	y={(0.5cm,0.8660254cm)},
	scale=2,
	vertex/.style={circle, fill=black, inner sep=1.5pt},
	]
	\coordinate (A) at (0,0);
	\coordinate (B) at (1,0);
	\coordinate (C) at (1,1);
	\coordinate (D) at (0,2);
	\coordinate (E) at (-1,2);
	\coordinate (F) at (-1,1);

	\draw (A) -- (B) -- (C) -- (D) -- (E) -- (F) -- (A);
	\draw[red,dashed,line width=2pt] (F) -- (A) -- (B);

	\node[vertex, label=below:{}] at (A) {};
	\node[vertex, label=below:{}] at (B) {};
	\node[vertex, label=right:{}] at (C) {};
	\node[vertex, label=above:{}] at (D) {};
	\node[vertex, label=above:{}] at (E) {};
	\node[vertex, label=left:{}] at (F) {};

	\node[] at (1.5,1) {\large$\Longrightarrow$};

	\coordinate (A') at (3,0);
	\coordinate (B') at (4,0);
	\coordinate (C') at (4,1);
	\coordinate (D') at (3,2);
	\coordinate (E') at (2,2);
	\coordinate (F') at (2,1);

	\draw (A') -- (B') -- (C') -- (D') -- (E') -- (F') -- (A');
	\draw[red,dashed,line width=2pt] (F') -- (A') -- (B');
	\draw[red,dashed,line width=2pt] (C') -- (D') -- (E');

	\node[vertex, label=below:{}] at (A') {};
	\node[vertex, label=below:{}] at (B') {};
	\node[vertex, label=right:{}] at (C') {};
	\node[vertex, label=above:{}] at (D') {};
	\node[vertex, label=above:{}] at (E') {};
	\node[vertex, label=left:{}] at (F') {};
	\end{tikzpicture}
	\]
	\caption{The square and hexagon properties, respectively. The dashed red edges belong to the convex rank test $\mathcal{E}$.}
	\label{fig:properties}
\end{figure}
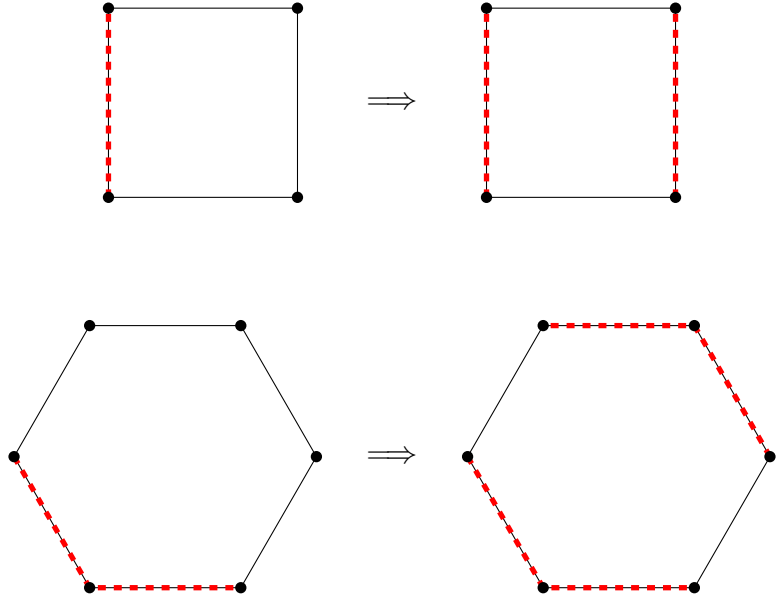

All the 2-faces of $\Pi_r$ are either squares or hexagons. Indeed, any 2-face of $\Pi_r$ corresponds to a flag of the form 
\[
\mathcal{V}\colon \varnothing \subset  V^1 \subset  \dots \subset \hat{V^k} \subset \dots \subset \hat{V^l} \subset \dots \subset V^{r}=[r]
\]
missing only subsets of size $k$ and $l$ with $k<l$. 

If $l>k+1$, the corresponding 2-face is a square: If $V^{k+1}\setminus V^{k-1}=\left\lbrace 1,2 \right\rbrace$ and $V^{l+1}\setminus V^{l-1} = \left\lbrace 3,4 \right\rbrace$, the 2-face has vertices and edges corresponding to the following flags
\begin{align*}
	\mathcal{V}[1|-] &\colon \varnothing \subset V^1 \subset  \dots \subset V^{k-1}\cup 1  \subset \dots \subset \hspace{0.5cm} \hat{V^{l}} \hspace{0.5cm} \subset \dots \subset [r] ,\\
	\mathcal{V}[-|3] &\colon \varnothing \subset V^1 \subset  \dots \subset \hspace{0.5cm} \hat{V^{k}}  \hspace{0.5cm} \subset \dots \subset {V^{l-1}}\cup 3 \subset \dots \subset [r] ,\\
	\mathcal{V}[1|3] & \colon \varnothing \subset V^1 \subset  \dots \subset V^{k-1}\cup 1  \subset \dots \subset {V^{l-1}}\cup 3  \subset \dots \subset [r], 
\end{align*}
with other flags defined similarly, arranged as pictured in Figure~\ref{fig:square}.

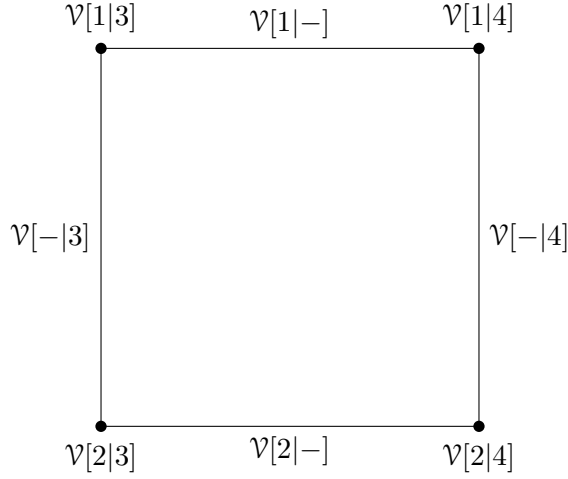
\begin{figure}
	\[
	\begin{tikzpicture}[scale=5,vertex/.style={circle, fill=black, inner sep=1.5pt}]
	\coordinate (AC) at (0,1);
	\coordinate (AD) at (1,1);
	\coordinate (BC) at (0,0);
	\coordinate (BD) at (1,0);

	\draw (AC) -- node[above] {$\mathcal{V}[1|-]$} (AD);
	\draw (BC) -- node[below] {$\mathcal{V}[2|-]$} (BD);
	\draw (AC) -- node[left] {$\mathcal{V}[-|3]$} (BC);
	\draw (AD) -- node[right] {$\mathcal{V}[-|4]$} (BD);

	\node[vertex, label=above:{$\mathcal{V}[1|3]$}] at (AC) {};
	\node[vertex, label=above:{$\mathcal{V}[1|4]$}] at (AD) {};
	\node[vertex, label=below:{$\mathcal{V}[2|3]$}] at (BC) {};
	\node[vertex, label=below:{$\mathcal{V}[2|4]$}] at (BD) {};
	\end{tikzpicture}
	\]
	\caption{A square 2-face of the permutohedron. The vertices and edges are labeled with the associated flags/preorders.}
	\label{fig:square}
\end{figure}

Otherwise if $l=k+1$, the corresponding 2-face is a hexagon: If $V^{k+2}\setminus V^{k-1}=\left\lbrace 1,2,3 \right\rbrace$, then the 2-face has vertices and edges corresponding to the following flags
\begin{align*}
	\mathcal{V}[12|3] & \colon \varnothing \subset V^1 \subset  \dots \subset V^{k-1} \subset \hspace{0.5cm} \hat{V^{k}} \hspace{0.5cm} \subset V^{k-1}\cup 1 \cup 2  \subset V^{k+2} \subset \dots \subset [r],\\
	\mathcal{V}[1|23] & \colon \varnothing \subset V^1 \subset  \dots \subset V^{k-1} \subset V^{k-1}\cup 1  \subset \hspace{0.6125cm} \hat{V^{k+1}} \hspace{0.6125cm} \subset V^{k+2} \subset \dots \subset [r],\\
	\mathcal{V}[1|2|3] & \colon \varnothing \subset V^1 \subset  \dots \subset V^{k-1} \subset V^{k-1}\cup 1  \subset V^{k-1}\cup 1 \cup 2  \subset V^{k+2} \subset \dots \subset [r],
\end{align*}
with other flags defined similarly, arranged as pictured in Figure~\ref{fig:hexagon}.

\begin{figure}
	\[
	\begin{tikzpicture}[
	x={(1cm,0cm)},
	y={(0.5cm,0.8660254cm)},
	scale=4,
	vertex/.style={circle, fill=black, inner sep=1.5pt},
	]
	\coordinate (AB) at (0,0);
	\coordinate (AC) at (1,0);
	\coordinate (CA) at (1,1);
	\coordinate (CB) at (0,2);
	\coordinate (BC) at (-1,2);
	\coordinate (BA) at (-1,1);

	\draw (AB) -- node[below] {$\mathcal{V}[1|23]$} (AC);
	\draw (AC) -- node[right] {$\mathcal{V}[13|2]$} (CA);
	\draw (CA) -- node[above right] {$\mathcal{V}[3|12]$} (CB);
	\draw (CB) -- node[above] {$\mathcal{V}[23|1]$} (BC);
	\draw (BC) -- node[left] {$\mathcal{V}[2|13]$} (BA);
	\draw (BA) -- node[below left] {$\mathcal{V}[12|3]$} (AB);

	\node[vertex, label=below:{$\mathcal{V}[1|2|3]$}] at (AB) {};
	\node[vertex, label=below:{$\mathcal{V}[1|3|2]$}] at (AC) {};
	\node[vertex, label=right:{$\mathcal{V}[3|1|2]$}] at (CA) {};
	\node[vertex, label=above:{$\mathcal{V}[3|2|1]$}] at (CB) {};
	\node[vertex, label=above:{$\mathcal{V}[2|3|1]$}] at (BC) {};
	\node[vertex, label=left:{$\mathcal{V}[2|1|3]$}] at (BA) {};

	\end{tikzpicture}
	\]
	\caption{A hexagonal 2-face of the permutohedron. The vertices and edges are labeled with the associated flags/preorders.}
	\label{fig:hexagon}
\end{figure}
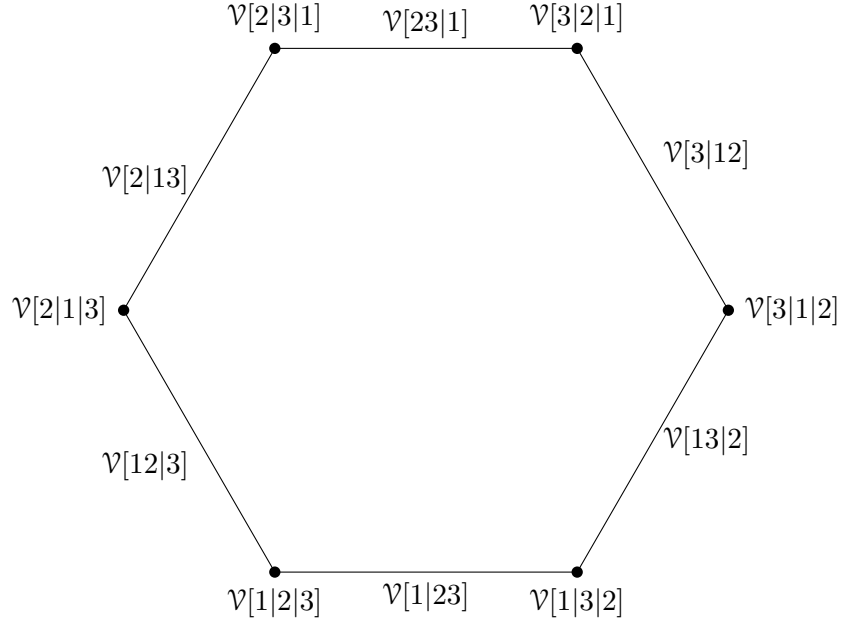

\subsubsection{Convexity of flag decorations on the permutohedron}

For any flag decoration $F:\operatorname{Fl}_r\to \operatorname{Fl}(\mathrm{M})$, let $\mathcal{E}(F)$ be the collection of edges of $\Pi_r$ defined as follows: the edge $(\mathcal{W},\mathcal{W}')\in \mathcal{E}(F)$ if there exists an ordered basis $B=(b_1,\dots,b_r)$ of $\mathrm{M}$ to which both $F(\mathcal{W})$ and $F(\mathcal{W}')$ are adapted simultaneously: For every $0\leq i \leq r$, $F(W^i)=\text{cl}_{\mathrm{M}}(b_j \mid j\in W^i)$ and  $F((W')^i)=\text{cl}_{\mathrm{M}}(b_j \mid j\in (W')^i)$. 

The main result of this section is the following:

\begin{theorem}\label{thm:permuto}
	For any flag decoration $F:\operatorname{Fl}_r \to \operatorname{Fl}(\mathrm{M})$, the induced collection of edges $\mathcal{E}(F)$ of $\Pi_r$ induces a fan coarsening $\Xi(F)$ of $\Sigma_r$.
\end{theorem}
\begin{proof}
	We prove that the collection $\mathcal{E}(F)$ satisfies the square and hexagon properties.

	Let $P(\mathcal{V})$ be a 2-face of $\Pi_r$ given by a flag 
	\[
	\mathcal{V}\colon \varnothing\subset V^1 \subset  \dots \subset \hat{V^k} \subset \dots \subset \hat{V^l} \subset \dots \subset [r] .
	\]

	Case $\square$. Suppose $P(\mathcal{V})$ is a square. If $(\mathcal{V}[1|3],\mathcal{V}[2|3])  \in \mathcal{E}(F)$, there exists an ordered basis $A=(a_1,\dots,a_r)$ adapted to both $F(\mathcal{V}[1|3])$ and $F(\mathcal{V}[2|3])$ simultaneously. To ease notation, write $F(S)=F(V^{l-1}\cup S)$ for any $S\subset \{3,4\}$. We have $F(3)$ and $F(4)$ are covered by $F(34)$ and cover $F(\varnothing)$. 
	
	Let $b_l \in F(4)\setminus F(\varnothing)$ and $b_{l+1}$ be a complement to $F(4)$ in $F(34)$, then $I=(a_1,\dots,b_{l},b_{l+1})$ is an independent set of $\mathrm{M}$. Both $F(\mathcal{V}[1|4])$ and $F(\mathcal{V}[2|4])$ simultaneously adapt to $I$ up to the $l+1$-th echelon. We can extend $I$ to a basis $B$ such that both $F(\mathcal{V}[1|4])$ and $F(\mathcal{V}[2|4])$ simultaneously adapt to $B$ since both flags agree above the $l+1$-th echelon. The alternative case, when $ (\mathcal{V}[1|3], \mathcal{V}[1|4])  \in \mathcal{E}(F)$ follows by a similar argument. 
	 
	Case $\hexagon$. Suppose $P(\mathcal{V})$ is a hexagon and without loss of generality both $( \mathcal{V}[1|2|3], \mathcal{V}[2|1|3])$ and $( \mathcal{V}[1|2|3],\mathcal{V}[1|3|2])$ belong to $\mathcal{E}(F)$. Let $A=(a_1,\dots,a_r)$ be a basis adapted to $\mathcal{V}[1|2|3]$ and $\mathcal{V}[2|1|3]$ simultaneously. To ease notation, write $F(S)=F(V^{k-1}\cup S)$ for any $S\subset \lbrace 1,2,3 \rbrace$. We have that:
	\begin{itemize}
		\item $F(1)\vee F(2)=F(12)$,
		\item $F(12)\vee F(13)=F(123)$ and $F(12)\wedge F(13)=F(1)$.
	\end{itemize}
	We can also deduce that:
	\begin{itemize}
		\item $F(2)\vee F(3)=F(23)$: If $F(3)=F(2)$ then $F(1)\vee F(3)=F(12)$ but $F(13)=F(3)\vee F(13)=F(3)\vee F(1)\vee F(13)=F(12)\vee F(13)=F(123)$, a contradiction.
		\item $F(2)\not\subset F(13)$: Else $F(12)\vee F(13)=F(1)\vee F(2) \vee F(13) =F(13)\neq F(123)$.
	\end{itemize}
	If $F(1)=F(3)$, let $b_1$ be a complement to $F(3)$ in $F(13)$, else if $F(1)\neq F(3)$, let $b_1\in F(1)$. In any case, $b_1\not\in F(23)$ since $F(13)\wedge F(23)=F(3)$ and $b_1\not\in F(3)$. Let $b_2\in F(2)\setminus F(\varnothing)$ and $b_3\in F(3)\setminus F(\varnothing)$, then $(b_1,b_2,b_3)$ are independent and $I=(a_1,\dots,a_{k-1},b_1,b_2,b_3)$ is an independent set to which $F(\mathcal{V}[3|2|1])$, $F(\mathcal{V}[2|3|1])$ and $F(\mathcal{V}[3|1|2])$ adapt to simultaneously up to the $k+2$-th echelon. We can extend $I$ to a basis $B$ such that the previous flags simultaneously adapt to $B$ since they agree above the $k+2$-th echelon. Hence both $(\mathcal{V}[2|3|1],\mathcal{V}[3|2|1])$ and $(\mathcal{V}[3|2|1],\mathcal{V}[3|1|2])$ belong to $ \mathcal{E}(F)$.
\end{proof}

\begin{corollary}\label{cor:single}
	If $(\mathcal{W}_0,\mathcal{W}_1)$, $(\mathcal{W}_1,\mathcal{W}_2) \in \mathcal{E}(F)$ are adjacent edges, then all three $F(\mathcal{W}_0), F(\mathcal{W}_1)$ and $ F(\mathcal{W}_2)$ adapt to the same ordered basis simultaneously. One obtains this by applying the argument from the previous proof twice, either in case $\square$ or case $\hexagon$.
\end{corollary}

In Section~\ref{sec:proofs} we will need a stronger sense of transitivity, we want all the flags in any maximal cone to adapt simultaneously to one ordered basis. 

Any collection of edges $\mathcal{E}$ of $\Pi_r$ defines a partition of the vertices of $\Pi_r$ by grouping vertices connected to each other by edges in $\mathcal{E}$. If $\mathcal{E}$ is a convex rank test, the classes of vertices with respect to this partition are in correspondence with maximal cones of $\Xi(\mathcal{E})$. Morton et al.~\cite{morton_convex_2009} prove that for any class $D$ of vertices there exists a unique partial order $\mathcal{U}$ on $[r]$ (not necessarily a total preorder) such that $D=\lbrace \mathcal{W} \mid \mathcal{U}\preceq \mathcal{W} \rbrace$, where $\preceq$ indicates order refinement. 

\begin{proposition}\label{prop:transitivity}
	If $F:\operatorname{Fl}_r \to \operatorname{Fl}(\mathrm{M})$ is a flag decoration, then for any maximal cone $\xi\in \Xi(F)$ written uniquely as $\xi=\cup_{\mathcal{W}\in D} C(\mathcal{W})$, there exists an ordered basis $B_\xi$ to which all the flags $\lbrace F(\mathcal{W})\rbrace_{\mathcal{W}\in D}$ adapt to. 
\end{proposition}
\begin{proof}
	Let $\mathcal{U}$ be the partial order corresponding to $D$. For each $i\in [r]$ define $U^{i-1} = \lbrace j\in [r] \mid j<_\mathcal{U} i \rbrace$ the set of strict predecessors of $i$ with respect to $\mathcal{U}$. For any lower order ideal $J$ with respect to $\mathcal{U}$, meaning that if $j\in J$ and $k<_{\mathcal{U}} j$ then $k\in J$, such that $U^{i-1}\subset J$ and $i\not\in J$, define $\del_i F(J) = F(J\cup i)\setminus F(J)$. Since $i\not\in J$, $\text{rk}_{\mathrm{M}}F(J\cup i)-\text{rank}_{\mathrm{M}} F(J)=1$. 
	
	First, we want to prove that if $U^{i-1} \subset K\subset J$ are lower ideals with respect to $\mathcal{U}$ not containing $i$, then $\del_i F(K) \subset \del_i F(J)$. It suffices to prove this when $J=K\cup j$ for some $j\not\in K$. Note that $j$ has to be incomparable to $i$ in $\mathcal{U}$, else $J$ would contain $i$. Since $F$ is a flag decoration, we have that $F(K\cup i)$ and $F(K\cup j)$ cover $F(K)$ and are covered by $F(K\cup i \cup j)$. Let $\mathcal{W}_i$ be a full flag in $D$ extending $K\subset K\cup i \subset K\cup i\cup j$ and let $\mathcal{W}_j$ be the adjacent full flag extending $K\subset K\cup j \subset K\cup i \cup j$. Note that $\mathcal{W}_j$ also belongs to $D$ since it refines $\mathcal{U}$. Since $(\mathcal{W}_i,\mathcal{W}_j)$ belongs to $\mathcal{E}(F)$, we have $F(K\cup i)\neq F(K\cup j)$, and in particular $F(K\cup i)\wedge F(K\cup j)=F(K)$ and $F(K\cup i)\vee F(K\cup j)=F(K\cup i \cup j)$. It follows that $\del_i F(K)= F(K\cup i)\setminus F(K) \subseteq F(K\cup i \cup j)\setminus F(K\cup j) = \del_i F(K\cup j)$, what we wanted.

	For each $i\in [r]$, choose $b_i \in \del_i F(U^{i-1})$. We claim that $B=(b_1,\dots,b_r)$ is an ordered basis to which the flags $\lbrace F(\mathcal{W})\rbrace_{\mathcal{W}\in D}$ adapt to. Let 
	\[
	\mathcal{W}\colon \varnothing \subset W^1 \subset \dots \subset W^r=[r]
	\] 
	be a full flag in $D$ with $W^i = \lbrace w_1,\dots, w_i \rbrace$. By the previous argument, we have that $b_{w_i} \in \del_{w_i} F(W^{i-1})$, given that $U^{w_i-1}\subset W^{i-1}$ as $\mathcal{U}\preceq \mathcal{W}$ and $W^{i-1}$ is a lower ideal with respect to $\mathcal{U}$ not containing $w_i$. Inductively we obtain that for any $i\in [r]$, $\text{cl}_\mathrm{M}(b_{w_1}\dots b_{w_i})=F(W^i)$, which proves that $B$ is a basis and more importantly that $F(\mathcal{W})$ adapts to $B$.   
\end{proof}

\subsection{Flag decorations of sweep polytopes}

Let $\boldsymbol{u}=\lbrace u_1,\dots,u_r \rbrace $ be a multiset of $r$ linear functionals on $\mathbb{R}^n$. A \emph{sweep} $\mathcal{W}_x$ of $\boldsymbol{u}$ determined by a point $x\in \mathbb{R}^n$ is the total preorder $\leq_x$ on $[r]$ defined by $i\leq_x j$ if and only if $\langle u_i,x  \rangle \geq \langle u_j, x \rangle$, for $i,j\in[r]$. Such a total preorder determines a unique flag $\mathcal{W}_x: \varnothing \subset W^1 \subset W^2 \subset \dots \subset W^k=[r]$ of subsets of $[r]$ given in the following way: Let $\lambda_1>\dots>\lambda_k$ be all the values of $\langle u_j, x \rangle$ for $j\in [r]$ in decreasing order, then $W^i= \lbrace j \in [r] \mid \langle u_j, x \rangle\geq \lambda_i \rbrace$ for $1\leq i \leq k$. We will use the term sweep interchangeably for the total preorder and the flag of sets. The set $\operatorname{Sw}(\boldsymbol{u})$ of sweeps of $\boldsymbol{u}$ is partially ordered by refinement, which we denote by $\preceq$.  

The multiset $\boldsymbol{u}$ has an associated \emph{sweep polytope} 
\[ 
Z(\boldsymbol{u})=\sum\limits_{i<j}\left[u_i,u_j\right],
\] 
whose normal fan $\Sigma({\boldsymbol{u}})\subseteq N_\mathbb{R}$ we call the \emph{sweep fan}. The face poset of the zonotope $Z(\boldsymbol{u})$ is anti-isomorphic to the face poset of $\Sigma({\boldsymbol{u}})$. By the work of Padrol and Philippe~\cite{padrol_sweeps_2024}, the poset $\operatorname{Sw}(\boldsymbol{u})$ is isomorphic to the face poset of $\Sigma({\boldsymbol{u}})$.

\begin{definition}
	A \emph{flag decoration} of the sweep polytope $Z(\boldsymbol{u})$ is a map of posets $F:{\operatorname{Sw}}(\boldsymbol{u}) \to \operatorname{Fl}(\mathrm{M})$ satisfying that for any sweep 
	\[
	\mathcal{W}\colon{\varnothing \subset W^1 \subset \dots \subset W^{k-1} \subset W^k=[r]}
	\]
	we have that 
	\[
	F(\mathcal{W})\colon{\varnothing \subset F(\mathcal{W})^1 \subset \dots \subset F(\mathcal{W})^{k-1} \subset F(\mathcal{W})^k=\mathrm{E}}
	\]
	has the same \emph{rank shape} as $\mathcal{W}$, meaning $\text{rk}_\mathrm{M} F(\mathcal{W})^i= |W^i|$ for every $1\leq i \leq k$.
\end{definition}

\subsubsection{Flag decorations of affine permutohedra} \label{sssec:combo_permuto} 

If $\boldsymbol{u}$ is an affinely independent set, then the sweep polytope $Z(\boldsymbol{u})$ is a \emph{affine permutohedron}, i.e. a permutohedron up to affine transformation. If $T^*: (\mathbb{R}^r)^*\to (\mathbb{R}^n)^*$ is the linear map $e_i^* \to u_i$, then the affine independence of $\boldsymbol{u}$ says that $T^*|_{\boldsymbol{1}^\perp}$ is an isomorphism between $\boldsymbol{1}^\perp$ and the span of the differences $u_i-u_j$. The map $T^*$ maps the permutohedron $\Pi_r\subset (\mathbb{R}^r)^*$ to $Z(\boldsymbol{u})\subset (\mathbb{R}^n)^*$. Dually, the adjoint $T$ of $T^*$ induces an isomorphism between $\mathbb{R}^n/L$ and $\mathbb{R}^r/\mathbb{R}\boldsymbol{1}$ where $L=\{ x \in \mathbb{R}^n \mid \langle u_1,x  \rangle=\dots = \langle u_r, x \rangle \}$. This way $T$ is an isomorphism of fans between $\Sigma_r$ in $\mathbb{R}^r/\mathbb{R}\boldsymbol{1}$ and $\Sigma(\boldsymbol{u})$ in $\mathbb{R}^n/L$. In particular, we have that $\operatorname{Sw}(\boldsymbol{u})\cong \operatorname{Fl}_r$.  

Consequently, the convexity of flag decorations on permutohedra is preserved after affine isomorphism: Theorem~\ref{thm:permuto} and Proposition~\ref{prop:transitivity} hold when $F:\operatorname{Fl}_r\to \operatorname{Fl}(\mathrm{M})$ is a flag decoration of the sweep polytope $Z(\boldsymbol{u})$ of an affinely independent set.

\subsubsection{General failure of convexity}

A theorem of Reading~\cite[Cor. 1.3]{reading_coarsening_2010} says that a collection of edges $\mathcal{E}$ of a zonotope $Z$ determines a coarsening $\Xi(\mathcal{E})$ of the normal fan of $Z$ if and only if $\mathcal{E}$ satisfies the \emph{polygon property}: for every $2k$-gonal 2-face $P$ of $Z$, if $k-1$ consecutive edges of $P$ belong to $\mathcal{E}$, then so do the opposite $k-1$ consecutive edges of $P$. 

One could hope that flag decorations on general sweep polytopes enjoy the same convexity properties, but this can fail as soon as we step outside affine permutohedra. Indeed, the polygon property can fail for flag decorations on octagons, the simplest 2 dimensional even polygon that is a sweep polytope but not combinatorially a permutohedron. 

\begin{example}\label{ex:non_convex_deco}
	Let $\boldsymbol{u}=\lbrace (0,2),(2,2),(4,2),(1,6) \rbrace$. Up to sign and scaling, the pairwise differences $u_i-u_j$ are $(1,0),(1,4),(-1,4),(-3,4)$ hence the sweep polytope $Z(\boldsymbol{u})$ is an octagon whose edge directions are parallel to those four vectors. Each vertex and edge corresponds to a sweep in $\operatorname{Sw}(\boldsymbol{u})$, arranged as pictured in Figure~\ref{fig:octagon}. 

	Let $\mathrm{M}$ be the rank 4 matroid on $\mathrm{E}=\lbrace a,b,c,d,e,f \rbrace$ whose circuits are $\lbrace c,d,e,f \rbrace$, $ \lbrace a,b,e,f \rbrace$ and $\lbrace a,b,c,d \rbrace$. The matroid $\mathrm{M}$ is isomorphic to the dual of the rank 2 matroid whose parallel classes are $\lbrace a,b \rbrace$, $\lbrace c,d \rbrace$ and $\lbrace e,f \rbrace$.

	\begin{figure}
		\[
		\begin{tikzpicture}[
			scale=0.6,
			vertex label/.style={
				font=\small
			},
			edge label/.style={
				font=\small,
				fill=white,
				inner sep=1.5pt
			}
		]

		\coordinate (v0) at (11,4);
		\coordinate (v1) at (8,8);
		\coordinate (v2) at (0,8);
		\coordinate (v3) at (-1,4);
		\coordinate (v4) at (0,0);
		\coordinate (v5) at (3,-4);
		\coordinate (v6) at (11,-4);
		\coordinate (v7) at (12,0);

		\draw
			(v0) -- node[midway, edge label, above right] {$34|2|1$} (v1)
			(v1) -- node[midway, edge label, above]       {$4|123$} (v2)
			(v2) -- node[midway, edge label, above left]  {$14|2|3$} (v3)
			(v3) -- node[midway, edge label, left=3pt]    {$1|24|3$} (v4)
			(v4) -- node[midway, edge label, below left]  {$1|2|34$} (v5)
			(v5) -- node[midway, edge label, below]       {$123|4$} (v6)
			(v6) -- node[midway, edge label, below right] {$3|2|14$} (v7)
			(v7) -- node[midway, edge label, right=3pt]   {$3|24|1$} (v0);

		\node[vertex label, above right] at (v0) {$3|4|2|1$};
		\node[vertex label, above right] at (v1) {$4|3|2|1$};
		\node[vertex label, above]       at (v2) {$4|1|2|3$};
		\node[vertex label, above left]  at (v3) {$1|4|2|3$};
		\node[vertex label, below left]  at (v4) {$1|2|4|3$};
		\node[vertex label, below left]  at (v5) {$1|2|3|4$};
		\node[vertex label, below right] at (v6) {$3|2|1|4$};
		\node[vertex label, right]       at (v7) {$3|2|4|1$};

		\foreach \v in {v0,v1,v2,v3,v4,v5,v6,v7}
			\fill (\v) circle (3pt);

		\end{tikzpicture}
		\]
		\caption{The sweep polytope $Z(\boldsymbol{u})$ for $\boldsymbol{u}=\lbrace (0,2),(2,2),(4,2),(1,6) \rbrace$. The vertices and edges are labeled with the associated sweeps.}
		\label{fig:octagon}
		
		\[
		\begin{tikzpicture}[
			scale=0.6,
			vertex label/.style={
				font=\small
			},
			edge label/.style={
				font=\small,
				fill=white,
				inner sep=1.5pt
			}
		]

		\coordinate (v0) at (11,4);
		\coordinate (v1) at (8,8);
		\coordinate (v2) at (0,8);
		\coordinate (v3) at (-1,4);
		\coordinate (v4) at (0,0);
		\coordinate (v5) at (3,-4);
		\coordinate (v6) at (11,-4);
		\coordinate (v7) at (12,0);

		\draw
			(v0) -- node[midway, edge label, above right] {$bf|ae|cd$} (v1)
			(v1) -- node[midway, edge label, above]       {$b|acdef$} (v2)
			(v2) -- node[midway, edge label, above left]  {$bc|ad|ef$} (v3)
			(v3) -- node[midway, edge label, left=3pt]    {$c|abd|ef$} (v4)
			(v4) -- node[midway, edge label, below left]  {$c|d|abef$} (v5)
			(v5) -- node[midway, edge label, below]       {$cdef|ab$} (v6)
			(v6) -- node[midway, edge label, below right] {$f|e|abcd$} (v7)
			(v7) -- node[midway, edge label, right=3pt]   {$f|abe|cd$} (v0);

		\draw[red,dashed,line width=2pt]
			(v0) -- (v1)
			(v1) -- (v2)
			(v2) -- (v3)
			(v3) -- (v4)
			(v4) -- (v5)
			(v6) -- (v7)
			(v7) -- (v0);

		\node[vertex label, above right] at (v0) {$f|b|ae|cd$};
		\node[vertex label, above right] at (v1) {$b|f|ae|cd$};
		\node[vertex label, above]       at (v2) {$b|c|ad|ef$};
		\node[vertex label, above left]  at (v3) {$c|b|ad|ef$};
		\node[vertex label, below left]  at (v4) {$c|d|ab|ef$};
		\node[vertex label, below left]  at (v5) {$c|d|ef|ab$};
		\node[vertex label, below right] at (v6) {$f|e|cd|ab$};
		\node[vertex label, right]       at (v7) {$f|e|ab|cd$};

		\foreach \v in {v0,v1,v2,v3,v4,v5,v6,v7}
			\fill (\v) circle (3pt);

		\end{tikzpicture}
		\]
		\caption{The flag decoration $F:\operatorname{Sw}(\boldsymbol{u})\to \operatorname{Fl}(\mathrm{M})$. The edges of $\mathcal{E}(F)$ are marked in red, and do not satisfy the polygon property.}
		\label{fig:octa_dec}
	\end{figure}

	Consider the flag decoration $F:\operatorname{Sw}(\boldsymbol{u})\to \operatorname{Fl}(\mathrm{M})$ depicted in Figure~\ref{fig:octa_dec}. One can check that the edges marked in this figure constitute $\mathcal{E}(F)$. However, the collection $\mathcal{E}(F)$ does not satisfy the polygon property, hence it does not define a coarsening of sweep fan $\Sigma({\boldsymbol{u}})$.

\end{example}

\section{Tropical bundles with trivial Chern class}\label{sec:proofs}

\subsection{Flag decorations from tropical bundles} In this section we prove Theorems~\ref{thm:trivial_chern} and~\ref{thm:analogous} by proving that any tropical bundle with trivial Chern class induces a flag decoration on a sweep polytope associated to the Chern roots of the bundle. In particular, if the Chern roots are independent, the sweep polytope in question is a affine permutohedron. 

Throughout this section, let $\mathfrak{E}$ be a tropical bundle of rank $r$ on a complete fan $\Delta\subseteq N_\mathbb{R}$ with piecewise linear map $\Phi:|\Delta|\to \Berg(\mathrm{M})$. Suppose that $\mathfrak{E}$ has trivial equivariant Chern class with Chern roots $\boldsymbol{u}=\lbrace u_1,\dots,u_r \rbrace$. Let 
\begin{align*}
	T:N_\mathbb{R} &\longrightarrow \mathbb{R}^r,\\
	x&\longmapsto (\langle u_1, x \rangle,\dots,\langle u_r, x \rangle),
\end{align*}
and let $U: N_\mathbb{R} \to \mathbb{R}^r/\mathbb{R}\boldsymbol{1}$ be its composition with the natural projection. For every cone $\delta\in \Delta$ there exists an ordered basis $B_\delta$ of $\mathrm{M}$ such that $\Phi|_\delta = \phi_{B_\delta} \circ T$. 
	
Consider the sweep fan $\Sigma({\boldsymbol{u}})$ in $N_\mathbb{R}$ and write $C(\mathcal{W})$ for the cone in $\Sigma({\boldsymbol{u}})$ associated to the sweep $\mathcal{W}$ of $\boldsymbol{u}$. If $x\in \delta \cap \operatorname{relint} C(\mathcal{W})$, the flag of flats of $\mathrm{M}$ determined by $\Phi(x)$, say $\mathcal{F}_x\colon \varnothing\subset F^1_x \subset \dots \subset F^{k-1}_x \subset F^k=\mathrm{E} $, is given by $F^i_x=\text{cl}_{\mathrm{M}}\left\lbrace b_j \mid j\in W^i \right\rbrace$, where $B_\delta=(b_1,\dots,b_r)$ is an ordered basis of $\Phi$ at $\delta$. By the continuity of $\Phi$, the flags $\mathcal{F}_x$ must stay constant in the relative interiors of any given $C(\mathcal{W})\in \Sigma({\boldsymbol{u}})$, denote this flag by $F(\mathcal{W})$. This was also observed by Payne~\cite[Prop. 4.8]{payne_toric_2008}. Indeed, if $x\in \operatorname{relint}C(\mathcal{W})$, then $\Phi(x)$ is given coordinatewise by
\begin{align}\label{eqn:coordinatewise}
	\Phi(x)_e = \langle u_{w_i}, x \rangle \text{ if }e\in F(\mathcal{W})^i \setminus F(\mathcal{W})^{i-1},
\end{align}
where $w_i\in W^i \setminus W^{i-1}$. This gives us a well-defined assignment of flags $F=F_{\mathfrak{E}}: \operatorname{Sw}(\boldsymbol{u}) \to \operatorname{Fl}(\mathrm{M})$ induced by $\mathfrak{E}$. Furthermore, this assignment of flags preserves the refinement order. In brief, we obtain the following lemma.

\begin{lemma}\label{lem:flag_deco}
	 The function $F: \operatorname{Sw}(\boldsymbol{u}) \to \operatorname{Fl}(\mathrm{M})$ induced by $\mathfrak{E}$ is a flag decoration of $Z(\boldsymbol{u})$. 
\end{lemma}

We are now ready to prove Theorem~\ref{thm:trivial_chern} and its companion Theorem~\ref{thm:analogous}.

\begin{proof}[Proof of Theorem~\ref{thm:trivial_chern}]
	Suppose that ${\boldsymbol{u}}=\{ u_1,\dots,u_r \} $ consists of affinely independent characters.
	
	The PL map $\Phi\otimes u$ of $\mathfrak{E} \otimes \mathfrak{O}(u)$ is given on every cone $\delta\in \Delta$ by $(\Phi\otimes u) |_\delta(x) = \phi_{B_\delta}  (T(x)+u(x)\boldsymbol{1})$. Since we are interested in the class of $\Phi$ up to tensoring with $u$, we may assume that $\ker T =\lbrace x\in N_\mathbb{R} \mid \langle u_1, x \rangle=\dots=\langle u_r, x \rangle \rbrace$. We can choose an integral injection $s: \mathbb{R}^r/\mathbb{R}\boldsymbol{1} \to \mathbb{R}^r$ such that
	\[
		\begin{tikzcd}
			N_\mathbb{R} \arrow[rr,"T"] \arrow[rd,"U"'] & &  \mathbb{R}^r\\
			{} & \mathbb{R}^r/\mathbb{R}\boldsymbol{1} \arrow[ru,"s"'] & 
		\end{tikzcd}
	\]
	for example, $s(x)=(x_1-x_r,\dots,x_{r-1}-x_r,0)$. For simplicity, we write $\Phi$ for $\Phi\otimes u$ with the $u\in M$ that makes the previous statement true. 
	
	Consider the sweep fan $\Sigma(\boldsymbol{u})$, by the discussion in Section~\ref{sssec:combo_permuto}, $Z(\boldsymbol{u})$ is a affine permutohedron of dimension $r-1$. Then by Lemma~\ref{lem:flag_deco} and Theorem~\ref{thm:permuto}, the flag decoration $F_{\mathfrak{E}}: \operatorname{Sw}(\boldsymbol{u}) \to \operatorname{Fl}(\mathrm{M})$ induces a fan ${\Xi}$ coarsening $\Sigma(\boldsymbol{u})$. Furthermore, for any maximal $\delta\in \Delta$, there exists a maximal $\xi \in {\Xi}$ containing $\delta$ since a codimension one wall of $\Sigma(\boldsymbol{u})$ intersects the relative interior of $\delta$ only if its corresponding edge of $Z(\boldsymbol{u})$ belongs to $\mathcal{E}(F)$. This means that $id: \Delta \to \Xi$ is a morphism of fans. This way, $\Phi$ defines a tropical bundle $\mathfrak{E}'$ on $\Xi$ given that for any maximal cone $\xi \in \Xi$, there exists an ordered basis $B_\xi$ such that $\Phi|_\xi = \phi_{B_\xi}\circ {T}$. To see this, take the ordered basis given by Proposition~\ref{prop:transitivity}, equality is guaranteed by equation \eqref{eqn:coordinatewise}. Hence, the following diagram commutes:
	\[
		\begin{tikzcd}
			\lvert \Delta \rvert \arrow[r,"id"] \arrow[rd,"\Phi "'] & \lvert \Xi \rvert \arrow[d,"\Phi"]\\
			{} & \Berg(\mathrm{M})
		\end{tikzcd}
	\]
	or equivalently $\mathfrak{E}\cong id^* \mathfrak{E}'$.
	
	Now, $\Sigma(\boldsymbol{u})$ and $\Sigma_r$ are isomorphic up to the lineality space $\ker U$ by means of $U$. Let $\tilde{\Xi}$ be the fan coarsening $\Sigma_r$ obtained from the edges corresponding to $\mathcal{E}(F)$, so the face complexes of $\Xi$ and $\tilde{\Xi}$ are identical. Additionally, ${U}:\Xi \to \tilde{\Xi}$ is a surjective morphism of fans. Let $\Psi: |\tilde{\Xi}| \to \Berg(\mathrm{M})$ be the PL map given on a maximal cone $\xi\in \tilde{\Xi}$ by $\Psi|_\xi = \phi_{B_\xi} \circ s$ where $B_\xi$ is chosen as in the definition of $\mathfrak{E}'$. By definition, $\Psi$ defines a tropical bundle $\mathfrak{F}$ on $\tilde{\Xi}$.
	
	Putting everything together we obtain that the following diagram commutes
	\[
		\begin{tikzcd}
			\lvert \Delta \rvert \arrow[r,"id"] \arrow[rdd,"\Phi"'] & \lvert \Xi \rvert \arrow[dd,"\Phi"'] \arrow[r,"{U}"] & \lvert \tilde{\Xi} \rvert \arrow[ldd,"\Psi"]\\
			\\
			{} & \Berg(\mathrm{M})
		\end{tikzcd}
	\]
	so indeed $\mathfrak{E} \cong id^* {U}^* \mathfrak{F}$. 
\end{proof}

Having now completed the proof of Theorem~\ref{thm:trivial_chern}, Corollary~\ref{cor:trivial} follows immediately.

\begin{proof}[Proof of Corollary~\ref{cor:trivial}]
	If $\mathfrak{E}$ is trivial, $\Phi=\phi_B\circ T$ for some basis $B$.  By definition, $\mathcal{E}(F)$ consists of all edges of $\Pi_r$, so $\Xi$ is trivial. Conversely, if $\Xi$ is trivial, then $\mathcal{E}(F)$ consists of all edges of $\Pi_r$, then by Proposition~\ref{prop:transitivity} there exists a basis $B$ such that $\Phi\otimes u =\phi_B\circ T$ for some character $u\in M$. 
\end{proof}

\begin{proof}[Proof of Theorem~\ref{thm:analogous}]
	Let $\mathcal{E}$ be a rank $r$ toric vector bundle on a complete toric variety $X_\Delta$ with trivial equivariant Chern class and affinely independent global Chern roots $\boldsymbol{u}=\lbrace u_1,\dots,u_r \rbrace$. Let $\mathfrak{E}$ be a DJS tropicalization with matroid representation $L: \mathrm{M}(\mathcal{E})\to E$. Let $\tilde{\Phi}: |\Delta| \to \tilde{\mathcal{B}}(E)$ and $\Phi: |\Delta| \to \Berg(\mathrm{M}(\mathcal{E}))$ be the PL maps associated to $\mathcal{E}$ and to $\mathfrak{E}$. By Theorem~\ref{thm:trivial_chern} there exist $\Xi$, $\Psi$ and $u$ that fit in the following diagram.
	\[
	\begin{tikzcd}
		\lvert\Delta\rvert \arrow[r,"{U}"] \arrow[rd,"\Phi\otimes u"'] \arrow[rdd,bend right=30,"\tilde{\Phi}\otimes u"'] & \lvert \Xi \rvert \arrow[d,"\Psi"] \arrow[dd,bend left=70, dashed, "\tilde{\Psi}"]\\
		{} & \Berg(\mathrm{M}) \arrow[d,hookrightarrow,"L"]\\
		{} & \tilde{\mathcal{B}}(E)
	\end{tikzcd}
	\]
	The existence of the PL map $\tilde{\Psi}: |\Xi| \to \tilde{\mathcal{B}}(E)$ in the dashed arrow is a consequence of~\cite[Thm. 2.14]{kaveh_tropical_2024}. If $\mathcal{F}$ is the toric vector bundle associated to $\tilde{\Psi}$, then $\pi_{{U}}^*\mathcal{F}$ is equivariantly isomorphic to $\mathcal{E}\otimes \mathcal{O}(\text{div}\chi^u)$.
\end{proof}

\subsection{Beyond independent Chern roots} \label{ssec:failure} The natural generalization of Theorem~\ref{thm:trivial_chern} that arises from our proof while dropping the independence assumption is the following:
\begin{quote}
	Suppose that $\mathfrak{E}$ has trivial Chern class and Chern roots $\boldsymbol{u}=\lbrace u_1,\dots,u_r \rbrace$. Let ${U}:N_\mathbb{R}\to \mathbb{R}^r/\mathbb{R}\boldsymbol{1}$ be the linear map given by ${U}(x)=(\langle u_1, x \rangle,\dots,\langle u_r, x \rangle)$ and let $\boldsymbol{v}=\lbrace e_1^*,\dots,e_r^*\rbrace|_{\text{im}({U})}$ be the restriction of the standard dual basis to $\text{im}({U})$. Then there exist the following:
	\begin{enumerate}
		\item a complete fan $\Xi$ on $\text{im}({U})$ coarsening the sweep fan $\Sigma(\boldsymbol{v})$,
		\item a tropical bundle $\mathfrak{F}$ of rank $r$ on $\Xi$ with PL map $\Psi: |\Xi|\to \Berg(\mathrm{M})$ and
		\item a character $u$,
	\end{enumerate}
	such that ${U}: \Delta \to \Xi$ is a morphism of fans and $\mathfrak{E}\otimes \mathfrak{O}(u)$ is isomorphic to $ {U}^* \mathfrak{F}$.
\end{quote}
However, even in this generality the statement is false. In fact, we can upgrade Example~\ref{ex:non_convex_deco} into an example of a tropical bundle with trivial Chern class that cannot factor through a fan coarsening of the sweep fan of its Chern roots. The difficulty lies in the fact that, in the absence of convexity for general flag decorations, the source fan $\Delta$ and the sweep fan $\Sigma(\boldsymbol{u})$ need not be compatible with each other.

\begin{figure}
		\begin{tikzpicture}[
			sweep ray/.style={
				draw=red!70!black,
				line width=2pt,
				dashed,
				-{Latex[length=2mm]}
			},
			sweep ray prime/.style={
				draw=red!70!black,
				line width=0.75pt,
				-{Latex[length=2mm]}
			},
			sigma ray/.style={
				draw=green!75!black,
				line width=1.2pt,
				-{Latex[length=2mm]}
			},
			sigma ray prime/.style={
				draw=green!75!black,
				line width=2.5pt,
				-{Latex[length=2mm]}
			},
			sweep label/.style={
				font=\small,
				text=red!70!black,
				fill=white,
				inner sep=1pt
			},
			sigma label/.style={
				font=\small,
				text=green!75!black,
				fill=white,
				inner sep=1pt
			}
		]

		\coordinate (O) at (0,0);


		\draw[sigma ray]
			(O) -- (3.58,1.79)
			node[sigma label,pos=0.74,above] {$\tau_1$};

		\draw[sigma ray]
			(O) -- (2.83,2.83)
			node[sigma label,pos=0.76,above left] {$\tau_2$};

		\draw[sigma ray]
			(O) -- (-3.58,1.79)
			node[sigma label,pos=0.72,above] {$\tau_3$};

		\draw[sigma ray]
			(O) -- (-4.00,0)
			node[sigma label,left] {$\tau_4$};

		\draw[sigma ray] (O) -- ( 3.88, 0.97); 
		\draw[sigma ray] (O) -- (-3.88,-0.97); 
		\draw[sigma ray] (O) -- (-3.20,-2.40); 
		\draw[sigma ray prime] (O) -- ( 0.00,-4.00); 
		\draw[sigma ray] (O) -- ( 3.88,-0.97); 


		\draw[sweep ray]
			(O) -- (3.20,2.40)
			node[sweep label,above right] {$\rho_1$};

		\draw[sweep ray]
			(O) -- (0,4.00)
			node[sweep label,above] {$\rho_2$};

		\draw[sweep ray]
			(O) -- (-3.88,0.97)
			node[sweep label,above left] {$\rho_3$};

		\draw[sweep ray]
			(O) -- (-3.88,-0.97)
			node[sweep label,below left] {$\rho_4\color{black}{ = }\color{green!75!black}{\tau_5}$};

		\draw[sweep ray]
			(O) -- (-3.20,-2.40)
			node[sweep label,below left] {$\rho_5\color{black}{ = }\color{green!75!black}{\tau_6}$};

		\draw[sweep ray prime]
			(O) -- (0,-4.00)
			node[sweep label,below] {$\rho_6\color{black}{ = }\color{green!75!black}{\tau_7}$};

		\draw[sweep ray]
			(O) -- (3.88,-0.97)
			node[sweep label,below right] {$\rho_7\color{black}{ = }\color{green!75!black}{\tau_8}$};

		\draw[sweep ray]
			(O) -- (3.88,0.97)
			node[sweep label,right] {$\rho_8\color{black}{ = }\color{green!75!black}{\tau_9}$};

		\fill (O) circle (1.7pt);
		\node[font=\small,below right] at (O) {};

		\end{tikzpicture}
		\caption{The rays of the fans $\Sigma(\boldsymbol{u})$ (red) and $\Delta$ (green). The dashed red rays are the rays of $\Sigma(\boldsymbol{u})$ corresponding to edges in $\mathcal{E}(F)$. All the rays of $\Sigma(\boldsymbol{u})$ but $\rho_6$ are dashed.}
		\label{fig:no_coarsening}
	\end{figure}
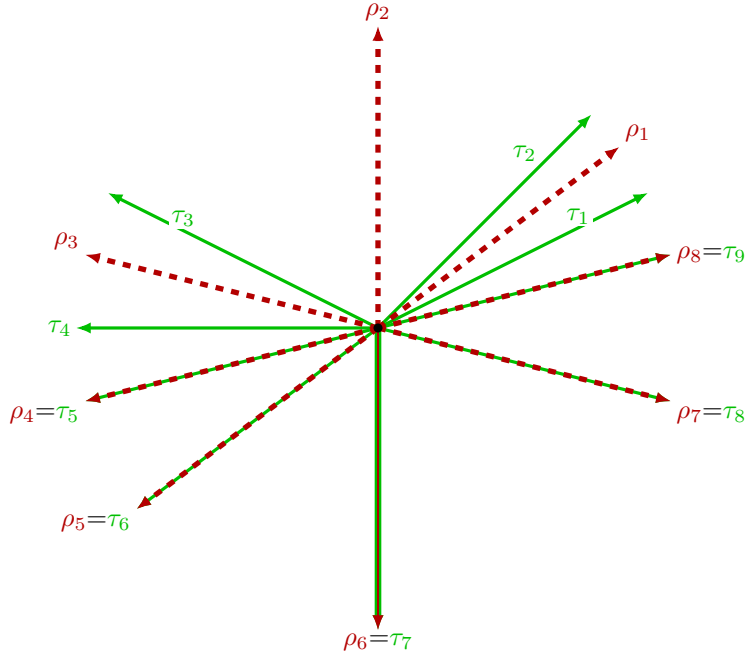

\begin{example}\label{ex:no_coarsening}
	Let $\boldsymbol{u}$ and $\mathrm{M}$ be as in Example~\ref{ex:non_convex_deco}. The eight rays, labeled $\rho_i$, of the sweep fan $\Sigma(\boldsymbol{u})$ are spanned by
	\begin{align*}
		f_1 &= (4,3), & f_2&= (0,1), & f_3&= (-4,1), & f_4&= (-4,-1),\\
		f_5&= (-4,-3), & f_6&= (0,-1), & f_7&= (4,-1), & f_8&= (4,1).
	\end{align*}
	Fix $N=\mathbb{Z}^2$ and let $\Delta$ be the complete fan on $N_\mathbb{R}$ with nine rays, labeled $\tau_i$, which are spanned by
	\begin{align*}
		g_1 &= (2,1), & g_2&= (1,1), & g_3&= (-2,1), & g_4&= (-1,0),
	\end{align*}
	as well as $g_5=f_4,g_6=f_5,g_7=f_6,g_8=f_7$ and $g_9=f_8$. The image of ${U}$ is 2-dimensional, so we may treat $N_\mathbb{R}=\text{im}({U})$. See Figure~\ref{fig:no_coarsening} to see how these two fans overlap.

	Let $\Phi: |\Delta| \to \Berg(\mathrm{M})$ be the PL map defined globally as
	\[
	\Phi=(\min(u_2,u_4),u_4,u_1,\min(u_1,u_2),\min(u_2,u_3),u_3).
	\]
	One can check that this map defines a tropical bundle on $\Delta$ and also induces the flag decoration $F$ from Example~\ref{ex:non_convex_deco}, which does not create a fan coarsening of $\Sigma(\boldsymbol{u})$. 

	Moreover, there cannot exist a nontrivial fan $\Xi$ coarsening $\Sigma(\boldsymbol{u})$ such that both every cone of $\Delta$ is contained in a cone of $\Xi$ and $\Phi$ defines a tropical bundle on $\Xi$. This is because of the following two opposing facts: Any such fan $\Xi$ would delete the rays $\rho_1,\rho_2$ and $\rho_3$, and that would require deleting $\rho_6$ because of the polygon property. But any fan on which $\Phi$ defines a tropical bundle requires $\rho_6$ as a ray, since there exists no basis of $\mathrm{M}$ that adapts to $\Phi$ the cones adjacent to $\rho_6$ simultaneously. Indeed, on the two cones adjacent to \(\rho_6=\tau_7\), the sets of adapted unordered bases are
	\begin{align*}
		\operatorname{cone}(\rho_5,\rho_6)&\colon\{ a,c,d,e \}, \{a,c,d,f\}, \{b,c,d,e\}, \{b,c,d,f\},\\
		\operatorname{cone}(\rho_6,\rho_7)&\colon\{a,c,e,f\}, \{a,d,e,f\}, \{b,c,e,f\}, \{b,d,e,f\},
	\end{align*}
	which are disjoint.
\end{example}

\section{Applications}\label{sec:examples}

This last section is dedicated to applications of our main theorems. Let us first prove Theorem~\ref{thm:rank4}.

\begin{proof}[Proof of Theorem~\ref{thm:rank4}]
	The nontrivial convex rank tests/fan coarsenings of $\Sigma_r$ are partially ordered by coarsening. By Studen\'y's~\cite{studeny_probabilistic_2005} classification of convex rank tests for $r\leq 4$, the coarsest nontrivial fans coarsening $\Sigma_4$ are all projective. If $X_\Delta$ admits no nonconstant maps to projective space, then it cannot admit a surjective map to a nontrivial coarsening of $\Sigma_4$, this includes all nontrivial coarsenings of lower dimensional permutohedral fans up to lineality. Thus, it cannot admit a nontrivial toric vector bundle of rank $r \leq 4$ with trivial Chern class and affinely independent Chern roots. 
\end{proof}

Next we study some classical examples of threefolds with no nonconstant maps to projective space through the lens of our theorem. 

\begin{example}[Fulton's threefold] \label{ex:fulton}
	Let $N=\mathbb{Z}^3$. Consider the following vectors in $N_\mathbb{R}$:
	\begin{align*}
		f_1&=(1,2,3), & f_2&=(1,-1,1), & f_3&=(-1,-1,1), & f_4&=(-1,1,1),\\
		f_5&=(1,1,-1), & f_6&=(1,-1,-1), & f_7&=(-1,-1,-1), & f_8&=(-1,1,-1),
	\end{align*}
	and let $\Delta$ be the fan whose maximal cones are:
	\begin{align*}
		\delta_1&=\operatorname{cone}(f_1,f_2,f_3,f_4), & \delta_2&=\operatorname{cone}(f_5,f_6,f_7,f_8), & \delta_3&=\operatorname{cone}(f_1,f_2,f_5,f_6),\\
		\delta_4&=\operatorname{cone}(f_2,f_3,f_6,f_7), & \delta_5&=\operatorname{cone}(f_3,f_4,f_7,f_8), & \delta_6&=\operatorname{cone}(f_1,f_4,f_5,f_8). 
	\end{align*}
	This example is due to Fulton~\cite[pp. 25-26]{fulton_introduction_2016}. Payne~\cite[Ex. 4.11]{payne_toric_2008} proves that the toric variety $X_\Delta$ does not admit any nontrivial toric vector bundles of rank $\leq 2$ yet admits a nontrivial toric vector bundle of rank 3 with nontrivial Chern class. Since $X_\Delta$ has no non-constant maps to projective space, it does not admit a nontrivial toric vector bundle of rank $r \leq 4$ with trivial Chern class and affinely independent Chern roots.
\end{example}

\begin{example}[Payne's threefold] \label{ex:payne}
	Let $N=\mathbb{Z}^3$. Consider the following vectors in $N_\mathbb{R}$:
	\begin{align*}
		f_1&=(1,2,3), & f_2&=(1,-1,2), & f_3&=(-1,-1,1), & f_4&=(-1,1,1),\\
		f_5&=(1,1,-1), & f_6&=(1,-1,-1), & f_7&=(-1,-1,-1), & f_8&=(-1,1,-1),
	\end{align*}
	and let $\Delta$ be the fan whose maximal cones are:
	\begin{align*}
		\delta_1&=\operatorname{cone}(f_1,f_2,f_3,f_4), & \delta_2&=\operatorname{cone}(f_5,f_6,f_7,f_8), & \delta_3&=\operatorname{cone}(f_1,f_2,f_5,f_6),\\
		\delta_4&=\operatorname{cone}(f_2,f_3,f_6,f_7), & \delta_5&=\operatorname{cone}(f_3,f_4,f_7,f_8), & \delta_6&=\operatorname{cone}(f_1,f_4,f_5,f_8). 
	\end{align*}
	This example is due to Payne~\cite[Ex. 4.13]{payne_toric_2008}, who proves that the toric variety $X_\Delta$ does not admit any nontrivial toric vector bundles of rank $\leq 3$. Since $X_\Delta$ has no non-constant maps to projective space, it does not admit a nontrivial toric vector bundle of rank $r\leq 4$ with trivial Chern class and affinely independent Chern roots.
\end{example}

We now turn to Theorem \ref{thm:existence}, but first we will need the following technical lemma:

\begin{lemma}\label{lem:graphic}
	Suppose $\Delta$ is a complete fan in $N_\mathbb{R}$ of dimension $n$ with no proper nontrivial coarsenings. For each of the codimension 1 faces of $\Delta$, choose a normal $w\in M_\mathbb{R}$ to it and let $W$ be the collection of such $w$. Let $\mathrm{M}(W)$ be the associated linear matroid. If $\mathrm{M}(W)$ is not a graphic matroid, then $\Delta$ admits no surjective map of fans $T:\Delta\to \Xi$ where $\Xi$ is a nontrivial coarsening of the $n$ dimensional permutohedral fan $\Sigma_{n+1}$. 
\end{lemma}
\begin{proof}
	Suppose $T:\Delta \to \Xi$ is such a map of fans. The linear map $T$ is an isomorphism. The fan $T^{-1}(\Xi)$ is a complete fan in $N_\mathbb{R}$ coarsening $\Delta$, but $\Delta$ has no proper coarsenings, so $\Delta = T^{-1}(\Xi)$. Let $V$ be the collection of normals to the walls of $\Xi$ and let $\mathrm{M}(V)$ be the associated matroid. Notice $V$ is comprised exclusively of elements of the form $e_i^*-e_j^*$ with $i\neq j$. In particular, $\mathrm{M}(V)$ is a graphic matroid since it is a deletion of a parallel extension of the matroid of the complete graph $K_{n+1}$. The dual map $T^*$ is a linear isomorphism, so $\mathrm{M}(V)$ and $\mathrm{M}(W)$ are isomorphic, but this is not possible since $\mathrm{M}(W)$ is not graphic by assumption.
\end{proof}

Finally, we study a family of complete toric varieties that admit no surjective maps to permutohedral fan coarsenings of the lower dimension. This family of varieties is due to Perling and Schröer~\cite[Sec. 6]{perling_vector_2016}. 

\begin{example}\label{ex:existence}
	For $n\geq 3$, let $N=\mathbb{Z}^n$ with standard basis $e_1,\dots,e_n$ and let $u\geq 1$ be an integer. Consider the following $2n+2$ vectors:
	\begin{align*}
		e &=e_n, & f_i&=e_i \text{ for }1\leq i <n, & f_n &= -\sum_{i=1}^{n-1} e_i,\\
		h &= -e_n, & g_i&=h-f_i \text{ for }1\leq i <n, & g_n &= uh - f_n,
	\end{align*}
	and let $\Delta_n(u)$ be the fan whose $\binom{n+1}{2}$ maximal cones are
	\begin{align*}
		\delta_i&= \operatorname{cone}(e,g_i,f_k)_{k\neq i} \text{ for }1\leq i\leq n,\\
		\delta_{ij}&= \operatorname{cone}(h,g_i,g_j,f_k)_{k\neq i,j}  \text{ for }1\leq i <j \leq n.
	\end{align*}
	One can determine that the pairs of maximal cones sharing a codimension 1 face are precisely $(\delta_i,\delta_j)$ for any distinct $i< j$, $(\delta_i,\delta_{ij})$ for any distinct $i,j$ and $(\delta_{ij},\delta_{ik})$ for any distinct $i, j, k$. Moreover, the collection $W$ of normals to the codimension 1 faces consists of the following vectors:
	\begin{align*}
		e_i^* - e_j^* & \text{ for }\delta_i\cap \delta_j,  i,j\neq n, & e_i^* & \text{ for }\delta_i\cap \delta_n,  i \neq n,\\
		e_i^* -e_j^* - e_n^* & \text{ for }\delta_i \cap \delta_{ij}, i,j\neq n, &  e_i^* - e_n^* & \text{ for }\delta_i \cap \delta_{in}, i\neq n,\\ ue_i^* + e_n^* & \text{ for }\delta_n \cap \delta_{in}, i\neq n, &
		e_j^*-e_k^* & \text{ for }\delta_{ij}\cap\delta_{ik}, j,k\neq n, \\ e_j^* &\text{ for }\delta_{in}\cap \delta_{ij}, i,j\neq n.
	\end{align*}
	
	\begin{proposition}
		The fan $\Delta_n(u)$ has no proper nontrivial coarsening. 
	\end{proposition}
	\begin{proof}
		By relabelling $\delta_i$ as $\delta_{0i}$, one can see that $\delta_{ij}$ and $\delta_{kl}$ share a codimension 1 face precisely when $|\{ ij \}\cap \{ kl \}|=1$. Therefore, the 1-skeleton $G$ of the intersection complex (note that this complex has inverted dimensions) of $\Delta_n(u)$ is isomorphic to the triangle graph $T(n+1)$. The codimension 2 faces of $\Delta_n(u)$ are in correspondence to regions of $G=T(n+1)$, which are all triangular. Suppose that $\Xi$ is a proper coarsening of $\Delta_n(u)$ given by a collection $\mathcal{E}$ of edges of $G$ corresponding to the edges to be merged. By Reading's ridge condition for fan coarsenings~\cite[Thm. 1.1]{reading_coarsening_2010}, if $\mathcal{E}$ defines a fan coarsening, then its intersection with any codimension 2 face needs to define a fan coarsening of the corresponding star fan. The 1-skeleton of any such star fan is triangular, thus a coarsening is either itself or trivial. Since $\Xi$ is a proper coarsening, $\mathcal{E}$ is non-empty, this forces all the triangles in $G$ containing that edge to also belong to $\mathcal{E}$, which in turn forces $\mathcal{E}$ to be all edges in $G$, i.e. for $\Xi$ to be the trivial coarsening. 
	\end{proof}

	\begin{proposition}
		The matroid $\mathrm{M}(W)$ determined by the normals $W$ of codimension 1 faces of $\Delta_n(u)$ is not a graphic matroid.
	\end{proposition}
	\begin{proof}
		Consider the five vectors $e_1^*, e_2^*,  e_1^*-e_n^*, -e_1^*+e_2^*-e_n^*, e_1^*-e_2^*-e_n^*$ in $W$. Deleting all other elements of $\mathrm{M}(W)$ and contracting $e_1^*$ gives a matroid isomorphic to $U_{2,4}$ which is a forbidden minor for both regular and graphic matroids, hence $\mathrm{M}(W)$ is not graphic.
	\end{proof}

	Combining the previous two propositions with Lemma~\ref{lem:graphic}, we obtain that $X_{\Delta_n(u)}$ admits no surjective morphism of fans to a nontrivial coarsening of the $n$-dimensional permutohedral fan, this includes all nontrivial coarsenings of lower dimensional permutohedral fans up to lineality. In particular, $\Delta_n(u)$ cannot admit a nontrivial toric vector bundle of rank $r\leq n+1$ with trivial Chern class and affinely independent Chern roots. Interestingly, by \cite[Prop. 6.1]{perling_vector_2016}, $X_{\Delta_n(u)}$ does not admit a nonconstant map to projective space if $u>1$, but $\Delta_n(1)$ is always projective.
\end{example}

\bibliographystyle{abbrvnat}
\bibliography{clean_trivial_chern}

\end{document}